\documentclass[11pt]{article}
\usepackage[a4paper,margin=27mm]{geometry}
\usepackage{amsmath,amssymb,amsthm,mathtools}
\usepackage{tikz}
\usetikzlibrary{arrows.meta,calc}
\usepackage{flafter}
\usepackage{placeins}
\usepackage[hidelinks]{hyperref}
\newtheorem{theorem}{Theorem}
\newtheorem{lemma}[theorem]{Lemma}
\newtheorem{proposition}[theorem]{Proposition}

\theoremstyle{remark}

\newcommand{\HH}{\mathbb H}
\newcommand{\RR}{\mathbb R}
\newcommand{\ip}[2]{\langle #1,#2\rangle}
\newcommand{\norm}[1]{\lVert #1\rVert}
\newcommand{\CELL}{\mathcal C}

\title{Sharp dimension bounds for right-angled\linebreak hyperbolic polytopes}
\author{Alexander Kolpakov}
\date{28 October 2026}

\begin{document}
\maketitle
\begin{abstract}
We determine the largest dimensions in which right-angled hyperbolic
polytopes of finite volume can exist. The dimension is at most eight
in general, and at most four when all vertices lie at infinity. Both
bounds are attained, improving the previously known upper bounds of
twelve and six. Our proofs use elementary properties of hyperbolic
polytopes together with bounds for configurations of points on spheres.
\end{abstract}

\section{Introduction}

Right-angled hyperbolic polytopes provide fundamental domains for
reflection groups. Potyagailo and
Vinberg~\cite{PV} construct finite-volume examples through dimension
eight. Dufour~\cite{Dufour} proves that their dimension is at most
twelve. For polytopes whose vertices are all ideal, the regular
right-angled $24$-cell gives an example in dimension four, while
the upper bound six was established in~\cite[Corollary~7]{Kolpakov}.
We close these two gaps. The compact case already has the sharp
dimension bound four; see~\cite[Section~2]{PV}.

A polytope is \emph{entirely ideal} if all its vertices are ideal.
Throughout, hyperbolic polytopes have finitely many facets and finite
volume, and the curvature is normalized to $-1$.
A right-angled polytope has dihedral angle $\pi/2$ at every ridge.
Finite and ideal vertices are allowed in any combination.

\begin{theorem}[All vertex types]\label{thm:general}
A finite-volume right-angled hyperbolic polytope has dimension at
most eight.
\end{theorem}

\begin{theorem}[Entirely ideal polytopes]\label{thm:ideal}
An entirely ideal finite-volume right-angled hyperbolic polytope has
dimension at most four.
\end{theorem}

Finite-volume right-angled examples in dimension eight are recalled
in~\cite{PV}. The regular ideal right-angled $24$-cell gives the
four-dimensional example. Thus both bounds are sharp.

\subsection{Outline of the proofs}
Both arguments begin with the same geometric observation. Projecting
a finite interior point to a facet of maximum unit-normal height
gives an interior point of that facet and decreases every inherited
height by a controlled factor. Repetition produces a low-dimensional
face whose outward normals determine a spherical code.

For Theorem~\ref{thm:general}, three projections in dimension nine
give a six-dimensional right-angled face whose facets determine
a spherical code with pairwise scalar products strictly below $1/4$.
An explicit cubic sum of tensor squares bounds the cardinality of
this code, and hence the number of facets, by $26$. This contradicts
the known minimum of $27$ facets in dimension six.

For Theorem~\ref{thm:ideal}, one projection in dimension five gives
an entirely ideal four-dimensional facet. Its facets determine
a code on $S^3$ with pairwise scalar products strictly below $1/2$,
so Musin's kissing-number theorem bounds its number of facets by $24$.
The minimum-facet theorem~\cite[Theorem~4]{Kolpakov} forces equality
and identifies the facet, up to hyperbolic isometry, with the regular ideal right-angled
$24$-cell. The heights at two opposite facets then contradict the
strict height bound from the projection, excluding the equality case.

The projection is proved in Section~\ref{sec:descent}; the two
contradictions follow in Sections~\ref{sec:code} and~\ref{sec:ideal}.
Figures~\ref{fig:ladder} and~\ref{fig:idealbranch} summarize the two
reductions.

\subsection{Previous work}
We recall three results used in the proofs.

\begin{proposition}[Six-dimensional facet minimum]\label{prop:27}
A six-dimensional finite-volume right-angled hyperbolic polytope
has at least $27$ facets.
\end{proposition}

This follows from Lemma~1 and Proposition~4 of
Dufour~\cite{Dufour}; see also~\cite[Section~4]{Alex}.

\begin{proposition}[Musin]\label{prop:musin}
A finite subset of $S^3\subset\RR^4$ with pairwise products at most
$1/2$ between distinct members has cardinality at most $24$.
\end{proposition}

This is the Main theorem of~\cite[p.~4]{Musin}.

\begin{proposition}[{\cite[Theorem~4]{Kolpakov}}]\label{prop:24}
An entirely ideal right-angled hyperbolic four-polytope has at least
$24$ facets. Equality holds only for the regular ideal right-angled
$24$-cell, up to hyperbolic isometry.
\end{proposition}

The last two propositions are needed only for Theorem~\ref{thm:ideal}.
The spherical bound used for Theorem~\ref{thm:general} is proved
directly in Section~\ref{sec:code}.

For a right-angled polytope in the Lorentz model, distinct outward unit facet normals have
nonpositive products, including facets that do not share a ridge.
As explained in~\cite[p.~42, \S2.1]{Vinberg}, nonobtuse ridge
angles imply this condition: nonadjacent supporting planes
do not intersect, and the off-diagonal Gram entries are nonpositive.
Faces of dimension at least three inherit right angles: every
intrinsic ridge has a finite relative-interior point, where the
incident walls are orthogonal. In the closure in the Klein model, a face has
vertices among those of the original polytope, so entire ideality is
inherited. Triangulation into finite or ideal simplices gives finite
volume.

\section{Facet projections}\label{sec:descent}

Use Lorentz signature $(-,+,\ldots,+)$, future unit timelike points,
and outward \emph{unit} spacelike normals. Thus
\[
 P=\{x\in\HH^n:\ip{x}{m}\le0\text{ for every facet normal }m\},
 \qquad \ip{m}{m'}\le0\quad(m\ne m').
\]
All heights below refer to this fixed unit-normal scaling. They are
the hyperbolic sines of distances to the corresponding walls.

\begin{lemma}[Farthest-facet projection]\label{lem:projection}
Let $q$ be a finite relative-interior point of a right-angled polytope
$P$, with $q^2=-1$. Set $t(m)=-\ip{q}{m}>0$, choose a facet normal
$n$ with largest height $t=t(n)$, and put
\begin{equation}\label{eq:projection}
 q'=\frac{q+tn}{\sqrt{1+t^2}}.
\end{equation}
Then $q'$ lies in the relative interior of the facet
$Q=P\cap n^\perp$. Each facet of $Q$ is induced by a facet of $P$
with unit normal $m_i$ orthogonal to $n$. This normal remains unit
in $n^\perp$ and has height
\begin{equation}\label{eq:newheight}
 0<t'_i=-\ip{q'}{m_i}
       =\frac{t(m_i)}{\sqrt{1+t^2}}
       \le\frac{t}{\sqrt{1+t^2}}<1.
\end{equation}
\end{lemma}
\begin{proof}
The numerator in~\eqref{eq:projection} has square $-1-t^2$ and
product zero with $n$. Its product with $q$ is $-1-t^2<0$, so
normalization gives a future unit point in $n^\perp$. For every
other facet normal $m$,
\begin{equation}\label{eq:allwall}
 \ip{q'}{m}=\frac{\ip{q}{m}+t\ip{n}{m}}{\sqrt{1+t^2}}<0.
\end{equation}
The first term in the numerator is negative and the second is
nonpositive. Thus $q'$ lies in the relative interior of $Q$.

Each facet of $Q$ is a ridge of $P$ shared with a unique other
facet, whose unit normal we denote by $m_i$.
Right angles give $\ip{n}{m_i}=0$, so restriction to $n^\perp$
preserves its unit norm. Formula~\eqref{eq:newheight} now follows
from the maximum choice of $t$.
\end{proof}

If all heights at $q$ are at most $A$, monotonicity of
$x/\sqrt{1+x^2}$ gives a new bound
$A/\sqrt{1+A^2}$. Start with the finite maximum $A_0$ at any finite
interior point, and let $A_j$ be the maximum height after $j$
projections. Repeating the construction intrinsically gives
\begin{equation}\label{eq:iterate}
 A_j\le\frac{A_0}{\sqrt{1+jA_0^2}}<\frac1{\sqrt j}
 \qquad(j\ge1),
\end{equation}
whenever the selected faces still have dimension at least three.
The formula follows by induction, or by adding one to the inverse
square of the previous bound.

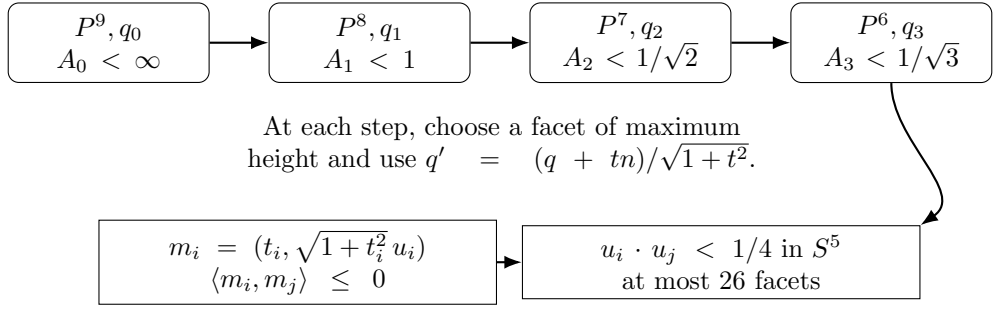
\begin{figure}[t]
\centering
\begin{tikzpicture}[>=Latex,font=\small,
 box/.style={draw,rounded corners,align=center,text width=2.4cm,
             minimum height=1.05cm}]
 \node[box] (q0) at(0,1.4) {$P^9,q_0$\\$A_0<\infty$};
 \node[box] (q1) at(3.45,1.4) {$P^8,q_1$\\$A_1<1$};
 \node[box] (q2) at(6.9,1.4) {$P^7,q_2$\\$A_2<1/\sqrt2$};
 \node[box] (q3) at(10.35,1.4) {$P^6,q_3$\\$A_3<1/\sqrt3$};
 \draw[->,thick] (q0)--(q1);
 \draw[->,thick] (q1)--(q2);
 \draw[->,thick] (q2)--(q3);
 \node[align=center,text width=11.7cm] at(5.2,.05)
  {At each step, choose a facet of maximum height and use
   $q'=(q+tn)/\sqrt{1+t^2}$.};
 \node[draw,align=center,text width=5cm] (split) at(2.5,-1.5)
  {$m_i=(t_i,\sqrt{1+t_i^2}\,u_i)$\\
   $\ip{m_i}{m_j}\le0$};
 \node[draw,align=center,text width=5cm] (code) at(8.1,-1.5)
  {$u_i\cdot u_j<1/4$ in $S^5$\\
   at most $26$ facets};
 \draw[->,thick] (q3.south) to[out=-90,in=35] (code.north east);
 \draw[->,thick] (split)--(code);
\end{tikzpicture}
\caption{Three successive facet projections and the resulting spherical code.
The arrows represent applications of the projection formulas.
The code has one element for each facet of the six-dimensional face,
so the bound of $26$ contradicts the minimum of $27$ facets.}
\label{fig:ladder}
\end{figure}

We will use the following translation from heights to spherical codes.
At a unit point $q$ of a $d$-dimensional face, choose $q$ as its time
axis and write the intrinsic unit normals as
\begin{equation}\label{eq:split}
 m_i=\bigl(t_i,\sqrt{1+t_i^2}\,u_i\bigr),
 \qquad u_i\in S^{d-1}\subset\RR^d.
\end{equation}
For distinct facets, $\ip{m_i}{m_j}\le0$ gives
\begin{equation}\label{eq:spatial}
 u_i\cdot u_j\le
 \frac{t_it_j}{\sqrt{(1+t_i^2)(1+t_j^2)}}.
\end{equation}
The vectors $u_i$ are distinct: if $u_i=u_j$, the Lorentz product of
the two unit normals would be positive. Consequently the number of
code points is exactly the number of facets of the selected face.

\section{A bound for spherical codes in dimension six}\label{sec:code}

\begin{lemma}\label{lem:26}
A finite set of unit vectors in $\RR^6$ whose distinct elements have
inner product strictly less than $1/4$ has at most $26$ members.
\end{lemma}
\begin{proof}
Consider the cubic polynomial
\begin{equation}\label{eq:polynomial}
 F(x)=(x+\tfrac12)^2(x-\tfrac14)
     =x^3+\tfrac34x^2-\tfrac1{16}.
\end{equation}
We express this polynomial as a sum of tensor inner products.

Use Euclidean tensor inner products. For a unit vector $u$
define the rank-two tensor
\[
 A(u)=u\otimes u-I/6
\]
and the symmetric rank-three tensors
\[
 S(u)_{abc}=u_a\delta_{bc}+u_b\delta_{ac}+u_c\delta_{ab},
 \qquad T(u)=u^{\otimes3}-S(u)/8.
\]
For $x=u\cdot v$, contraction gives
\begin{equation}\label{eq:kernels}
 \ip{A(u)}{A(v)}=x^2-\tfrac16,
 \qquad \ip{T(u)}{T(v)}=x^3-\tfrac38x.
\end{equation}
Each cross contraction of a cubic tensor with $S$ is $3x$; the
remaining contraction is
\[
 \ip{S(u)}{S(v)}=(3\cdot6+6)x=24x.
\]
The scalar identity therefore becomes
\begin{align*}
 F(u\cdot v)&=\tfrac1{16}+\tfrac38(u\cdot v)\\
 &\quad+\tfrac34\ip{A(u)}{A(v)}+\ip{T(u)}{T(v)}.
\end{align*}
Thus, for $u_1,\ldots,u_N$,
\begin{align}
 \sum_{i,j}F(u_i\cdot u_j)
 &=\frac{N^2}{16}+\frac38\norm{\sum_i u_i}^2\notag\\
 &\quad+\frac34\norm{\sum_i A(u_i)}^2
   +\norm{\sum_i T(u_i)}^2\notag\\
 &\ge\frac{N^2}{16}.\label{eq:sos}
\end{align}

On the other hand, $F(1)=27/16$, and $F(x)\le0$ for $x<1/4$.
The diagonal terms and~\eqref{eq:sos} give
\[
 \frac{N^2}{16}\le\sum_{i,j}F(u_i\cdot u_j)\le\frac{27N}{16},
\]
so $N\le27$. If $N=27$, equality throughout forces each
off-diagonal polynomial value to vanish. Since $u_i\cdot u_j<1/4$
strictly, the only available root is $-1/2$. Four such vectors would
already give
\[
 \norm{u_1+u_2+u_3+u_4}^2=4+12(-\tfrac12)=-2,
\]
which is impossible. Hence $N\le26$.
\end{proof}

\begin{proof}[Proof of Theorem~\ref{thm:general}]
Suppose first that a right-angled nine-polytope exists. Three
projections from a finite interior point produce a six-face
$R$ and a relative-interior point at which every facet height is
strictly less than $1/\sqrt3$, by~\eqref{eq:iterate}.
The inherited face is right-angled and has finite volume.
For its intrinsic normals,~\eqref{eq:spatial} now gives
\[
 u_i\cdot u_j<\frac14\qquad(i\ne j)
\]
in $S^5\subset\RR^6$. Lemma~\ref{lem:26} gives at most $26$
facets for $R$, contradicting Proposition~\ref{prop:27}.
Every higher-dimensional polytope has a nine-face inheriting finite
volume and right angles, giving the same contradiction.
\end{proof}

\section{The ideal bound from one projection}\label{sec:ideal}

We first record a property of the regular ideal right-angled $24$-cell.
In the standard Lorentz model of the regular ideal right-angled
$24$-cell, its outward unit normals are
\begin{equation}\label{eq:24normals}
 (1,\epsilon e_i+\delta e_j),\qquad
 1\le i<j\le4,\quad\epsilon,\delta\in\{1,-1\},
\end{equation}
and the center is $(1,0)$.

\begin{lemma}\label{lem:center}
No future unit timelike point of the regular ideal right-angled
$24$-cell has all its unit-normal heights strictly less than one.
More precisely, if every height is at most one, the point is the
center and all the heights equal one.
\end{lemma}
\begin{proof}
The normals in~\eqref{eq:24normals} come in opposite spatial pairs
$(1,r),(1,-r)$. Write the point as $(a,w)$, with $a>0$ and
$a^2-\norm w^2=1$. The two corresponding heights are
$a-r\cdot w$ and $a+r\cdot w$. If both are at most one, their
sum gives $a\le1$. Future unit timelikeness gives $a\ge1$.
Thus $a=1,w=0$, and every height equals one. Strict inequalities
on all heights are impossible.
\end{proof}

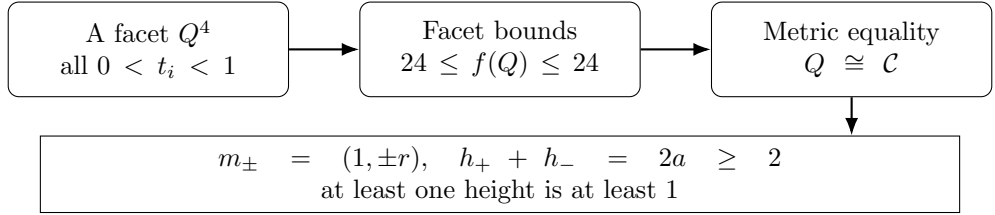
\begin{figure}[htbp]
\centering
\begin{tikzpicture}[>=Latex,font=\small,
 box/.style={draw,rounded corners,align=center,text width=3.45cm,
             minimum height=1.25cm}]
 \node[box] (facet) at(0,0)
  {A facet $Q^4$\\all $0<t_i<1$};
 \node[box] (count) at(4.65,0)
  {Facet bounds\\$24\le f(Q)\le24$};
 \node[box] (metric) at(9.3,0)
  {Metric equality\\$Q\cong\CELL$};
 \draw[->,thick] (facet)--(count);
 \draw[->,thick] (count)--(metric);
 \node[draw,align=center,text width=11.9cm] (pair) at(4.65,-1.65)
  {$m_\pm=(1,\pm r),\quad h_++h_-=2a\ge2$\\
   at least one height is at least $1$};
 \draw[->,thick] (metric.south)--(metric.south |- pair.north);
\end{tikzpicture}
\caption{The ideal case. Musin's spherical bound and the facet
theorem~\cite[Theorem~4]{Kolpakov} identify the projected facet with the regular ideal
$24$-cell. The heights at two opposite facets then contradict
the strict height bound from the projection.}
\label{fig:idealbranch}
\end{figure}

\begin{proof}[Proof of Theorem~\ref{thm:ideal}]
Assume an entirely ideal right-angled five-polytope exists. A single
farthest-facet projection from a finite interior point gives a
four-dimensional facet $Q$ at whose relative-interior point all heights satisfy
$0<t_i<1$. It inherits finite volume, entire ideality and right angles.
By~\eqref{eq:spatial}, the associated unit vectors in $S^3$ have
pairwise products strictly below $1/2$. Proposition~\ref{prop:musin}
bounds its facet number by $24$; Proposition~\ref{prop:24} gives the
reverse bound and identifies $Q$ with the regular ideal $24$-cell.
Lemma~\ref{lem:center} contradicts the strict heights $t_i<1$.

There is no entirely ideal right-angled five-polytope. In any higher
dimension a five-face would inherit all three hypotheses and give
the same contradiction. The ideal right-angled $24$-cell attains
dimension four.
\end{proof}

\FloatBarrier
\section{Supplementary verification}\label{sec:reproduce}

Exact Python checks verify the $24$-cell coordinates
in~\eqref{eq:24normals}, their Gram products and opposite pairs,
and the tensor identities~\eqref{eq:kernels}--\eqref{eq:sos}.
The polynomial identities are checked over $\mathbb Q$ modulo the
unit-sphere equations. The code and verification records are
available in the accompanying repository~\cite{Repository}.

Selected Lean~4 modules verify the internal logical deductions
for the projection, the cubic bound and the dimension bounds.
The geometric statements, face inheritance and cited theorems
are supplied as explicit premises.

\section*{Acknowledgments}
The author used OpenAI's GPT-6-astra and GPT-6.1-sol as research and coding
assistants to investigate proof ideas, help strengthen the results, improve
the exposition, and generate and audit the Python and Lean~4 code
accompanying this paper. The author reviewed the resulting mathematics
and code and takes responsibility for the final content.

\end{document}